\documentclass[amsart,11pt]{article}
\usepackage{amsmath,setspace,amssymb,amsfonts,amsthm,latexsym,amssymb,manyfoot}
\usepackage{mathrsfs}
\numberwithin{equation}{section}
\newtheorem{thm}{Theorem}

\newtheorem{lem}{Lemma}

\theoremstyle{definition}

\newcommand{\fip}{\varphi}
\newcommand{\ee}{\varepsilon}

\numberwithin{equation}{section} \numberwithin{lem}{section}
\numberwithin{thm} {section} \numberwithin{rem} {section}
\numberwithin{prop} {section} \numberwithin{cor} {section}

\newcommand{\Om}{\Omega}
\newcommand{\om}{\omega}
\newcommand{\R}{{\mathbb R}}
\newcommand{\Rd}{\mathbb{R}^2}
\newcommand{\de}{\partial}

\newcommand{\calP}{\mathcal P}
\newcommand{\dive}{\mathrm{div}\;}
\newcommand{\Pda}{\mathcal P(d\alpha)}
\newcommand{\calE}{\mathcal E}
\newcommand{\vn}{v_n}
\makeatletter \catcode`@=11
\newbox\tr@tto
\setbox\tr@tto=\hbox{{\count0=0\dimen0=-,9pt\dimen1=1,1pt\loop\ifnum
    \count0<11 \advance \count0 by1 \vrule width.51pt height\dimen1
    depth\dimen0\kern-0.17pt\advance\dimen0 by-0.05pt\advance\dimen1
    by0.1pt\repeat \loop\ifnum\count0<21\advance \count0 by1 \vrule
    width.6pt height\dimen1 depth\dimen0\kern-0.2pt \advance\dimen0
    by-0.1pt\advance\dimen1 by 0.05pt\repeat}}
\def\medint{\displaystyle\copy\tr@tto\kern-10.4pt\int}

\newcommand{\loc}{_{\fam 0 loc}}
\newcommand{\dmedint}{{}\hbox
{\vrule height 2,95pt depth -2,2pt width 6pt}\kern-0.94em }
\newcommand{\tmedint}{{}\hbox
{\vrule height 2,7pt depth -2,3pt width 5pt}\kern-8,5pt }
\newcommand{\smedint}{{}\hbox
{\vrule height 2,1pt depth -1,7 pt width 3pt}\kern-6,2pt }
\newcommand{\ssmedint}{{}\hbox
{\vrule height 1,7pt depth -1,3 pt width 3pt}\kern-6,2pt }
\newcount\ContatoreCostanti
\def\cnum#1{\global\advance\ContatoreCostanti by 1 \xdef#1{C_{\number \ContatoreCostanti}}#1}
\begin{document}
\title{Blow-up analysis
for some mean field equations involving probability measures from statistical hydrodynamics}
\author{T.~Ricciardi\thanks{Corresponding author} and G.~Zecca\\
\small{Dipartimento di Matematica e Applicazioni ``R.~Caccioppoli"}\\
\small{Universit\`{a} di Napoli Federico II}\\
\small{Via Cintia - 80126 Napoli - Italy}\\
\small{E-mail: tonricci@unina.it - g.zecca@unina.it }}
\date{}
\maketitle
\begin{abstract}
Motivated by the mean field equations with probability measure derived by Sawada-Suzuki and by Neri in
the context of the statistical mechanics description
of two-dimensional turbulence, we study the semilinear elliptic equation with probability measure:
\begin{equation*}
-\Delta v=\lambda\int_IV(\alpha,x,v)e^{\alpha v}\,\Pda
-\frac{\lambda}{|\Omega|}\iint_{I\times\Om}V(\alpha,x,v)e^{\alpha v}\,\Pda dx,
\end{equation*}
defined on a compact Riemannian surface.
This equation
includes the above mentioned equations of physical interest as special cases.
For such an equation we study the blow-up properties of solution sequences.
The optimal Trudinger-Moser inequality is also considered.

\bigskip


\end{abstract}
\section{Introduction}
\label{sec:intro}
Motivated by several mean field equations recently derived in the context of
Onsager's statistical mechanics description of turbulence \cite{Onsager}, we study concentrating
sequences of solutions to the
following equation:
\begin{equation}
\label{genmf}
-\Delta v=\lambda\int_IV(\alpha,x,v)e^{\alpha v}\,\Pda
-\frac{\lambda}{|\Omega|}\iint_{I\times\Om}V(\alpha,x,v)e^{\alpha v}\,\Pda dx,
\end{equation}
where $\Omega$ is a compact two-dimensional orientable Riemannian manifold
without boundary, $I=[-1,1]$, $\calP\in\mathcal M(I)$ is a Borel measure,
$v\in H^1(\Om)$ is a function normalized by $\int_\Om v=0$, $\lambda>0$ and
$V(\alpha,x,v)$ is a \emph{functional} satisfying the condition
$\alpha V(\alpha,x,v)\ge0$,
as well as suitable bounds which will be specified below.
\par
A typical special case of physical interest is given by
\begin{equation*}
\label{VSS}
V(\alpha,x,v)=V_1(\alpha,x,v)=\frac{\alpha}{\int_\Om e^{\alpha v}\,dx},
\end{equation*}
in which case equation~\eqref{genmf} reduces to the mean field equation derived by Sawada and Suzuki in \cite{SawadaSuzuki}:
\begin{equation}
\label{SawadaSuzuki}
-\Delta v=\lambda\int_I\alpha \left(\frac{e^{\alpha v}}
{\int_\Om e^{\alpha v}\,dx}-\frac{1}{|\Om|}\right)\,\calP(d\alpha).
\end{equation}
In order to relate our results to the literature,
we note that under the further assumption $\calP=\delta_1$, the Dirac concentrated at $\alpha=1$,
equation~\eqref{SawadaSuzuki} reduces to the well known mean field equation
\begin{equation}
\label{standmfe}
-\Delta v=\lambda\left(\frac{e^v}{\int_\Om e^v\,dx}-\frac{1}{|\Om|}\right)
\end{equation}
extensively studied in recent years. See, e.g., \cite{sbook} and the references therein
for results and applications of \eqref{standmfe} to physics, biology and geometry.
Assuming instead that
\begin{equation}
\label{twomass}
 \calP=t\delta_1+(1-t)\delta_{-1},
\end{equation}
equation~\eqref{SawadaSuzuki}
reduces to the mean field sinh-Gordon type equation
derived in \cite{JM,PL}.
Several blow-up results for \eqref{SawadaSuzuki}--\eqref{twomass}
have been obtained in recent years
by Ohtsuka and Suzuki in \cite{os1,os2}, and applied to derive the best constant
for the corresponding Trudinger-Moser inequality.
A construction of two-sided blow up solutions was obtained in \cite{EW}.
A blow-up analysis for \eqref{SawadaSuzuki} is contained in \cite{ORS},
and the best constant for the corresponding Trudinger-Moser inequality will appear
in \cite{RS}.
\par
Another special case of physical interest, which is the main motivation to this work, is
given by
\begin{equation}
\label{neripotential}
V(\alpha,x,v)=V_2(\alpha,x,v)=\frac{\alpha}{\iint_{I\times\Om}e^{\alpha v}\,\Pda}.
\end{equation}
In this case, equation~\eqref{genmf} reduces to the mean field equation derived by Neri~\cite{ne}:
\begin{equation}
\label{NE}
- \Delta v = \lambda\,\frac{\int_I \alpha (e^{\alpha v}- \frac{1}{|\Om|} \int_\Om e^{\alpha v}dx)\,\calP (d\alpha)}
{\iint_{I\times \Om} e^{\alpha v}\,\calP (d\alpha)dx}.
\end{equation}
An existence result for solutions to equation \eqref{NE} under Dirichlet boundary conditions
was also obtained in \cite{ne}.
In view of the results in \cite{ORS}, it is natural to study the concentrating
sequences of solutions to \eqref{NE}.
Actually, since both equations are motivated by the same physical problem,
it is natural to compare these equations and to seek their common features
as well as their differences.
Taking this point of view, we study the general mean field equation \eqref{genmf} and
we derive blow-up properties which are common to equation~\eqref{SawadaSuzuki}
and equation~\eqref{NE}.
On the other hand, we will show that from the point of view of the Trudinger-Moser inequality,
the two equations exhibit different properties.
Indeed, while equation~\eqref{SawadaSuzuki} leads to an improved best constant
with respect to \eqref{standmfe}, somewhat unexpectedly
such a situation does \emph{not} occur for equation~\eqref{NE}.
\par
We organize this article as follows. In Section~\ref{sec:results} we state our main blow-up results
for \eqref{genmf}, namely Theorem~\ref{thm:firstblowup} and Theorem~\ref{thm:secondblowup}.
We note that some results, such as \eqref{limiteq}, are new even for
equation~\eqref{SawadaSuzuki}. In Section~\ref{sec:proofs} we carry out the blow-up analysis.
Although we follow the approach in \cite{ORS},
based on the consideration of measures defined on the product space
$I\times\Om$, some technical lemmas are stated under weaker and more natural assumptions.
In Section~\ref{sec:neri} we apply our results to the special cases of physical interest.
We prove some results specific to Neri's equation~\eqref{NE},
particularly in relation to the residual vanishing property and the optimal Trudinger-Moser inequality,
see Theorem~\ref{thm:residual} and Theorem~\ref{thm:TM}, respectively.
\par
\emph{Notation.}
In what follows, we denote by $C$ a general constant whose value may change from line to line.
For all $p\in\Om$ we denote by $\delta_p\in\mathcal M(\Om)$ the Dirac measure centered at $p$.
For all $\alpha\in I$ we denote by $\delta_\alpha\in\mathcal M(I)$ the Dirac measure centered at $\alpha$.
We denote by $dx$ the volume element on $\Om$ and by $|\Om|$ the volume of $\Om$.
When the integration variable is clear from the context, for simplicity we omit it.
\section{Main results}
\label{sec:results}
We define
\[
\mathcal E=\left\{v\in H^1(\Om):\ \int_\Om v=0\right\}.
\]
and we make the following assumptions on the functional $V$.
\begin{enumerate}
\item[(V1)]
$(\mathrm{sign}\,\alpha)\,V(\alpha,x,v)\ge0$ for all $(\alpha,x,v)\in I\times\Om\times\calE$;
\item[(V2)]
$\sup_\mathcal E\|V(\alpha,x,v(x))\|_{L^\infty(I\times\Om)}\le C_1$
for some constant $C_1>0$;
\item[(V3)]
$\iint_{I\times\Om}|V(\alpha,x,v)|e^{\alpha v}\,\Pda dx\le C_2$ for some constant $C_2>0$.
\end{enumerate}
We consider solution sequences $\{ v_n\}$, $\lambda_n\rightarrow \lambda_0$ to
\begin{equation}
\label{genmfn}
\left\{
\begin{split}
&-\Delta v_n = \lambda_n \int_I\left(V(\alpha,x,v_n)e^{\alpha v_n}-\frac{1}{|\Om|}
\int_\Om V(\alpha,x,v_n)e^{\alpha v_n}\,dx\right)\,\Pda\\
&\int_\Om v_n=0.
\end{split}\right.
\end{equation}
Following the approach of Brezis and Merle~\cite{bm}, see also Nagasaki and Suzuki~\cite{NS},
we begin by proving that the blow-up set
for concentrating solutions is finite and that a ``minimum mass" is necessary for blow-up to occur.
We define the blow-up sets:
\[
\mathcal S_\pm =\{ p\in\Om :\exists p_{\pm,n }\rightarrow p : v_n(p_{\pm,n}) \rightarrow \pm\infty)\}
\]
and denote $\mathcal S = \mathcal S_+\cup \mathcal S _-.$
We define the measures $\nu_{\pm,n} \in \mathcal M(\Om)$ by setting
\begin{equation}
\label{nin}
\nu_{\pm,n}=\lambda_n\int_{I_\pm}|V(\alpha,x,\vn)|e^{\alpha\vn}\,\Pda
\end{equation}
where $I_+ = [0, 1]$ and $I_- = [-1, 0)$. Since in view of (V3) we have
$\int_\Om \nu_{\pm,n}\leqslant C_2\lambda_n$, we may assume  that
$\nu_{\pm,n} {\stackrel{*}\rightharpoonup} \nu_\pm$ for some measure $\nu_\pm \in\mathcal M(\Om)$.
\begin{thm}
\label{thm:firstblowup}
Assume (V1)--(V2)--(V3).
Let ${v_n}$ be a solution sequence to \eqref{genmfn} with $\lambda_n\rightarrow \lambda_0$. Then, the following alternative holds.
\begin{enumerate}
\item[i)] Compactness: $\limsup_{n\rightarrow\infty}\|v_n\|_\infty <+\infty$. There exist a solution $v \in\mathcal E$
to \eqref{genmf} with $\lambda=\lambda_0$
and a subsequence $\{v_{n_k}\}$ such that $v_{n_k}\rightarrow v$
in $\mathcal E$.
\item[ii)] Concentration: $\limsup_{n\rightarrow\infty}\|v_n\|_{L^\infty} =+\infty$.
 The sets $\mathcal S_\pm$ are finite and $\mathcal S = \mathcal S_- \cup\mathcal S_+ \neq \emptyset$.
For some $s_\pm\geqslant 0 $, $s_\pm\in L^1(\Om)$ we have
    \[
\nu_\pm = s_\pm dx+ \sum_{p\in\mathcal S_\pm} n_{\pm,p}\delta_p
\]
with $n_{\pm,p}\geqslant 4\pi$ for all $p\in \mathcal S$. Moreover, there exist
$v\in H^1_{\loc}(\Om\setminus \mathcal S)$, $k\in L^\infty(I\times\Om)$ and $c_0\in\R$ such that $v_n\rightarrow v $ in
$H^1_{\loc}(\Om\setminus \mathcal S)$ and
\begin{equation}
\label{limiteq}
\left\{\begin{split}
&- \Delta v = \lambda_0 \int_Ik(\alpha,x)e^{\alpha v} \calP (d\alpha)
+\sum_{p\in\mathcal S_+} n_{+,p}\delta_p - \sum_{p\in\mathcal S_-} n_{-,p}\delta_p - c_0\qquad\mbox{ in }\Om,\\
&\int_\Om v=0.
\end{split}\right.
\end{equation}
\end{enumerate}
\end{thm}
Under stronger assumptions on $V$, the blow-up results may be refined.
Following \cite{ORS} we consider measures defined on the product space $I\times \Om$.
We assume that $V$ does not depend on $x$, namely $V=V(\alpha,v)$
and
\begin{enumerate}
\item[(V0)]
$\nabla_xV(\alpha,v)=0$.
\end{enumerate}
We also strengthen assumptions (V2)--(V3) above as follows:
\begin{enumerate}
\item[(V2')]
$\sup_\mathcal E\|\alpha^{-1}V(\alpha,v)\|_{L^\infty(I)}\le C_1'$
for some constant $C_1'>0$;
\item[(V3')]
$\iint_{I\times\Om}|\alpha^{-1}V(\alpha,v)|e^{\alpha v}\,\Pda dx\le C_2'$ for some constant $C_2'>0$.
\end{enumerate}
For every fixed $\alpha \in I$ we define $\mu_\alpha^n (dx) \in \mathcal M(\Om)$ by setting
\begin{equation}
\label{mu-gn}
\mu_\alpha^n (dx)=\lambda_n\frac{V(\alpha,v_n)}{\alpha}e^{\alpha v_n}\,dx.
\end{equation}
We consider the sequence of measures $\mu_n=\mu_n(d\alpha dx)\in \mathcal M (I\times \Om) $ defined by
\begin{equation}\label{mu-n}
\mu_n(d\alpha dx)= \mu^n_\alpha (dx) \calP(d\alpha)
=\lambda_n\frac{V(\alpha,v_n)}{\alpha}e^{\alpha v_n}\,dx\Pda.
\end{equation}
In view of (V3'), for large values of $n$ we have:
\begin{eqnarray*}
\mu_n(I\times\Om)=\iint_{I\times\Om}\mu^n_\alpha (dx) \calP(d\alpha)
\leqslant C_2'(\lambda_0+1).
\end{eqnarray*}
Hence, upon extracting a subsequence, we may assume that \begin{equation}\label{mu}
\mu_n\stackrel{*}{\rightharpoonup}\mu  \mbox{ for some Borel measure } \mu\in\mathcal M (I\times \Om).
\end{equation}
In the next result we describe some properties of $\mu$.
\begin{thm}
\label{thm:secondblowup}
Suppose that $V$ satisfies (V0)--(V1)--(V2')--(V3').
Let $v_n$ be a solution sequence to \eqref{genmfn} with $\lambda_n\to\lambda_0$.
The following properties hold.
\begin{enumerate}
\item[(i)]
The singular part of $\mu$ has a ``separation of variables'' form:
\begin{equation}\label{zita}
\mu(d\alpha dx)= \sum_{p\in \mathcal S}\zeta _p(d\alpha) \delta_p(dx) + r(\alpha,x) \calP(d\alpha) dx.
\end{equation}
Here, $\zeta_p\in\mathcal M(I)$ and
$r\in L^1 (I\times\Om)$.
\item[(ii)]For every $p\in\mathcal S$ the following relation is satisfied
\begin{equation}\label{iii}
8\pi \int_I \zeta_p(d\alpha) = \left[ \int_I\alpha \zeta_p(d\alpha)\right]^2.
\end{equation}
\item[(iii)]For every $p\in\mathcal S$ it holds
\[
\int_{I_{\pm}}|\alpha| \zeta_p(d\alpha) =n_{\pm,p}\qquad \qquad\int_{I_\pm}|\alpha|r(\alpha,x)\calP(d\alpha)=s_\pm(x),
\]
where $n_{\pm,p}$ and $s_\pm(x)$ are as in Theorem~\ref{thm:firstblowup}. Moreover, for every $p\in \mathcal S_\pm \setminus \mathcal S_\mp$
\[
\int_{I_\mp} |\alpha|\zeta_p(d\alpha)=0.
\]
\end{enumerate}
\end{thm}
\section{A blow-up analysis}
\label{sec:proofs}
We begin by recalling some preliminary results.
We first provide an extension of a key result from \cite{bm} to the case of potentials defined on product spaces,
following the approach in \cite{ORS}.
We actually weaken the assumptions in \cite{ORS} and derive a somewhat more natural formulation.
Let $D\subset \Rd$ be a bounded domain and for every $a\in \R $ let $a^+$ be the positive part of $a$, $a^+=\max\{a,0\}$.
\begin{lem}
\label{lem:bm}
Let $(u_n)$ be a solution sequence to
\[
-\Delta u_n= \int_{I_+} W_n(\alpha,x)e^{\alpha u_n}\calP(d\alpha) \qquad \mbox{in }D,
\]
where $W_{\alpha,n}\geq 0$ verifies $\|\int_{I_+}W_{\alpha,n} \calP(d\alpha)\|_{L^p(D)}\leqslant C$,
$p\in (1,\infty]$ and $\|u_n^+\|_{L^1(D)}\leqslant C$. Suppose that for every $n\in\mathbb N$ we have
\begin{equation}\label{bme2}
\iint_{D\times I_+}W_n(\alpha,x)e^{\alpha u_n} \calP(d\alpha)dx \leqslant\ee_0<
\frac{4\pi}{p'},
\end{equation}
where $p' = p/(p -1)$ is the conjugate exponent to p. Then, $\{u_n^+ \}$ is bounded in $L^\infty_{loc}(D)$.
\end{lem}
\begin{proof}
Without loss of generality we may assume that $D=B_R.$ We split $u_n$ as $u_{1n}+u_{2n}$ where $u_{1n}$ is the solution of
\begin{equation}\label{bme}
\left\{
\begin{split}
&-\Delta u_{1n}= \int_{I_+}W_n(\alpha,x)e^{\alpha u_n} \mbox{ in }B_R\\
&u_{1n}=0 \mbox{ on } \partial B_R\\
\end{split}
\right.
\end{equation}
so that $\Delta u_{2n}=0$ in $B_R.$ By the mean value theorem for harmonic functions, \eqref{bme} and assumption \eqref{bme2} we have
\[
\| u ^+_{2n}\|_{L^\infty(B_{R/2})}\leqslant C\|u_{2n}^+\|_{L^1(B_R)} \leqslant C \left[\|u^+_{n}\|_{L^1(B_{R})} +\|u_{1n}\|_{L^1(B_{R})}\right]\leqslant C
\]
We define
\[
\fip_n = \int_{I_+} W_n(\alpha,x) e^{\alpha u_n} \calP ( d \alpha)
\]
so that, by \eqref{bme2} we have
\begin{equation}\label{fin}
\|\fip_n\|_{L^1(B_R)}\leqslant \ee_0 < \frac{4\pi}{p'}.
\end{equation}
By \cite{bm}, Theorem~1, for any $\delta\in (0,4\pi)$ we have
\[
\int_D \exp{\left[ \frac{(4\pi-\delta) |u_{1n}(x)|}{\|\fip\|_{L^1}}\right]dx } \leqslant \frac{4\pi^2}{\delta}(\mbox{diam } D)^2.
\]
\noindent Moreover, since $\ee_0<\frac{4\pi}{p'}$ there exists $\delta_0 \in (0,4\pi )$ such that $\ee_0=\frac{4\pi-\delta_0}{p'}$. Then, for  $\bar \delta\in (0,\delta_0)$, using \eqref{fin} we have
\begin{equation}\label{u1n}
\int_{B_{R/2}} \exp{\left[ (p'+ \eta) |u_{1n}(x)|\right]dx } \leqslant
\int_{B_{R/2}} \exp{\left[ \frac{(4\pi-\bar \delta) |u_{1n}(x)|}{\|\fip\|_{L^1({B_{ R}})}}\right]dx } \leqslant C
\end{equation}
where $\eta=\frac{\delta_0-\bar \delta}{4\pi-\delta_0}p'$. Hence, the sequence $\{e^{|u_{1n}|}\}$ is bounded in $L^{p'+\eta}(B_R)$ so that the sequence $\{e^{u_{n}^+}\}$ is bounded in $L^{p'+\eta}(B_{R/2})$ for some $\eta>0.$
On the other hand,
\begin{equation*}
\begin{split}
\int_{B_{R/2}}\left| \int_{I_+} W_n(\alpha,x) e^{\alpha u_n} \calP(d\alpha) \right|^r dx&\leqslant \int_{B_{R/2}}e^{ r u_n^+} \left( \int_{I_+} W_n(\alpha,x) \calP(d\alpha) \right) ^r dx\\
&\leqslant \left( \int_{B_{R/2}}e^{\frac {p r}{p-r}u_n^+} \right)^{\frac{p-r}{p}}\left(\int_{B_{R/2}} \left(\int_{I_+} W_n(\alpha,x) \calP(d\alpha) \right)^p dx\right)^{\frac rp}\\
&= \|e^{u_n^+}\|^r_{L^{\frac{pr}{p-r}}(B_{R/2})} \left\|\int_{I_+} |W_{\alpha,n}| \calP(d\alpha)\right \|_{L^p (B_{R/2})}^r. \\
\end{split}
\end{equation*}
If we choose $r\in (1,p)$ in order to have $pr/(p-r)= p'+\eta,$ by \eqref{bme} and the elliptic estimates we see that $u_{1n}$ is bounded in $L^\infty(B_{R/4})$.
Therefore $\{u_n^+\}$ is bounded in $L^\infty(B_{R/4})$.
\end{proof}
Now we recall the following result for equations defined on manifolds obtained in \cite{ORS} (see also \cite{os}). Let $(\Om, g)$ be a Riemannian surface. We consider solution sequences $\{u_n\}$ to the equation
\begin{equation}\label{20}
-\Delta u_n = \int_{I_+}W_n(\alpha,x) e^{\alpha u_n} \calP (d\alpha) + f_n \qquad \mbox{ on } \Om
\end{equation}
and set
\[
\sigma_n = \int_{I_+}W_n(\alpha,x) e^{\alpha  u_n}\calP(d\alpha).
\]
We weaken the assumptions in \cite{ORS} by assuming uniform boundedness of
\[
\|W_n(\alpha,x)\|_{L^p(D;L^1(I_+))}
\]
with respect to $n$. We recall that $\|W_n(\alpha,x)\|_{L^\infty(D;L^1(I_+))}\le C$
was assumed in \cite{ORS}.
\begin{lem}
\label{lem:bmonmflds}
Let $\{u_n\}$ be a solution sequence to \eqref{20} where
\[\left \|\int_{I_+}W_n(\alpha,x)\calP (d\alpha) \right \|_{L^p} \leqslant C,
 \]
$W_n(\alpha,x)\ge0$,
$\|f_n\|_{\infty } \leqslant C$ and $\|u_n^+\|_1 \leqslant C$.
Suppose that $\sigma_n {\stackrel{*}{\rightharpoonup}}\sigma$ and $\sigma(\{x_0\}) < 4\pi/p'$ for some $ x_0 \in \Om$.
Then, there exists a neighborhood $\widetilde U\subset \Om $ of $x_0$ such that
\[
\limsup_{n\rightarrow \infty } \|u_n^+\|_{L^\infty(\widetilde U)}<+\infty.
\]
\end{lem}
\begin{proof}
Let $(U,\psi)$ a local isothermal chart such that $\psi({x_0})=0$,
$g=e^{\xi(X)}(dX_1^2+dX_2^2)$.
Then, $u_n(X)=u_n(\psi^{-1}(X))$ satisfies
\[
-\Delta_Xu_n=\left(\int_{I_+}W(\alpha,n)e^{\alpha u_n}\,\Pda+f_n\right)e^{\xi}
\qquad\mbox{in\ }D=\psi(U).
\]
Let $h_n$ be defined by
\begin{align*}
&-\Delta h_n=f_n e^{\xi}\ \mbox{in }D,
&&h_n=0\ \mbox{on\ }\de D.
\end{align*}
Then, $\|h_n\|_{L^\infty(D)}\le C$ and $\widetilde u_n=u_n-h_n$
satisfies
\[
-\Delta\widetilde u_n=e^\xi\int_{I_+}W(\alpha,x)e^{h_n}e^{\alpha\widetilde u_n}\,\Pda
\qquad\mbox{in\ }D.
\]
On the other hand, setting $\widetilde W_n(\alpha,x)=e^\xi W(\alpha,x)e^{h_n}$
we have
\begin{align*}
\|\widetilde W_n(\alpha,x)\|_{L^p(D;L^1(I_+))}
\le\|e^\xi e^{h_n}\|_\infty\|W_n(\alpha,x)\|_{L^p(D;L^1(I_+))}\le C.
\end{align*}
Moreover,
\begin{align*}
\|\widetilde u_n^+\|_{L^1(D)}\le\|u_n^+\|_{L^1(\Om)}+|D|\|h_n\|_{L^\infty(D)}\le C
\end{align*}
and
\[
\int_De^\xi\int_{I_+W(\alpha,x)e^{\alpha h_n}}e^{\alpha\widetilde u_n}\,\Pda dX=\sigma_n(U).
\]
In view of the assumption, there exists $U'\subset U$, $x_0\in U'$ such that
\[
\iint_{I_+\times U'}W(\alpha,x)e^{\alpha u_n}\,\Pda dx\le\ee_0<\frac{4\pi}{p'}.
\]
In view of Lemma~\ref{lem:bm}, $\widetilde u_n$ is bounded in $L^\infty_{\mathrm{loc}}(\psi^{-1}(U'))$.
Taking $\widetilde U\Subset U'$ we conclude the proof.
\end{proof}
We can now prove our first result.
\begin{proof}[Proof of Theorem~\ref{thm:firstblowup}]
We denote by $G = G(x, y)$ the Green's function associated to $-\Delta$ on $\Om$. Namely, $G$ is defined by
\begin{equation}\label{green}
\left\{
\begin{array}{ll}
-\Delta_x G(x,y)= \delta_y -\frac{1}{|\Om|}& \\
&\\
\int_\Om G(x, y)dx = 0.\\
\end{array}
\right.
\end{equation}
\bigskip
For every solution $v_n$ to \eqref{genmfn} we define
\[
 u_{\pm,n}(x)= G\star \nu_{\pm,n} (x) = \int_\Om G(x,y) \nu_{\pm,n}(y)dy
 \]
where $\nu_{\pm,n}$ is defined in \eqref{nin}. Then, $v_n = u_{+,n}- u_{-,n}$.
We observe that $u_{\pm,n}$ is uniformly bounded below.
Indeed, let $A>0$ be such that $G(x,y)\ge-A$ for all $x,y\in\Om$.
Then,
\begin{align*}
u_{\pm,n}(x)=\int_\Om G(x,y)\nu_{\pm,n}(y)\,dy\ge -A\int_\Om\nu_{\pm,n}(y)\,dy
\ge-AC_2\lambda_n\ge-AC_2(\lambda_0+1).
\end{align*}
In this sense, we say that $u_{+,n}$ is the ``positive part'' of $v_n$
and $u_{-,n}$ is the ``negative part'' of $v_n$.
Furthermore, in view of Assumption~(V1), the functions $u_{\pm,n}$ satisfy the Liouville system:
\begin{equation}
\label{Liouvillepm}
\left\{
\begin{split}
&-\Delta  u_{\pm,n}=\lambda_n \int_{I_\pm}|V(\alpha,x,v_n)|e^{|\alpha|(u_{\pm,n}-u_{\mp,n})}\,\Pda-c_{\pm,n}  \\
&\int_\Om  u_{\pm,n}\,dx=0
\end{split}
\right.
\end{equation}
where
\[
c_{\pm,n}=\frac{\lambda_n}{|\Om|}\iint_{I_\pm\times\Om}|V(\alpha,x,v_n)|e^{\alpha v_n}\,\Pda.
\]
We check that the equations in \eqref{Liouvillepm} satisfy the assumptions of Lemma~\ref{lem:bmonmflds} with
$W_n(\alpha,x)=|V(\alpha,v_n(x))|e^{-\alpha u_{-,n}(x)}$, $p=\infty$ and $f_n=c_{+,n}$.
To this end, we note that in view of Assumption~(V2), we have, for every $\alpha\in I_+$:
\begin{equation*}
V(\alpha,x,v_n)e^{-\alpha u_{-,n}}\le V(\alpha,x,v_n)e^{AC_2(\lambda_0+1)}\le C_1e^{AC_2(\lambda_0+1)}.
\end{equation*}
Therefore, setting $W_n(\alpha,x)=|V(\alpha,v_n(x))|e^{-\alpha u_{-,n}(x)}$ we have
\[
0\le W_n(\alpha,x)\le C_1 e^{AC_2(\lambda_0+1)}.
\]
Furthermore,
\begin{align*}
c_{+,n}\le\frac{\lambda_n}{|\Om|}\iint_{I_+\times\Om}|V(\alpha,x,v_n)|e^{u_{+,n}-u_{-,n}}\,\Pda
\le C_2\frac{\lambda_0+1}{|\Om|}.
\end{align*}
\par
Let
\[
\mathcal S_{u_+ }= \{ p\in \Om : \nu_+(\{p \})\geqslant 4\pi\}.
\]
Since $\nu_{+,n}(\Om)= \int_\Om \nu_{+,n}dx \leqslant C_2\lambda_n$ and $\nu_{+,n}(\Om)\rightarrow \nu_+(\Om)$
then $\nu_+(\Om) \leqslant C_2(\lambda_0+1) <\infty$, so that $\sharp\mathcal S_{u_+} <\infty$.
\par
\textit{Claim~1.} If $\mathcal S_{u_+} =\emptyset $, then Alternative~(i) holds.

Indeed, if  $\mathcal S_{u_+} =\emptyset $ holds, then
in view of Lemma~\ref{lem:bmonmflds} with $p=+\infty$ and of the compactness of $\Om$ we have
\[
\limsup_{n\rightarrow +\infty} \|u^+_{+,n}\|_{L^{\infty}(\Om)}<+\infty.
\]
Then, by elliptic estimates,
\[
\limsup_{n\rightarrow \infty}\|u_{+,n}^+\|_{W^{2,r}(\Om)}<+\infty, \qquad r\in[1,+\infty),
\]
and therefore we may extract a subsequence $\{u_{+,n_k}\}$ such that $u_{+,n_k} \rightarrow    u_+$, for some $ u_+\in \mathcal E$. Similarly, if $S_{u_-}= \emptyset$ then there exists a subsequence $u_{-,n_k} \rightarrow    u_-$, for some $ u_-\in \mathcal E$, where
\[
\mathcal S_{u_- }= \{ p\in \Om : \nu_-(\{p \})\geqslant 4\pi\}.
\]
We conclude that if $\mathcal S_{u_+}\cup\mathcal S_{u_-}=\emptyset$,
then $v_n\to v=u_+-u_-$ in $\mathcal E$. Claim~1 is established.
\par
\textit{Claim~2.}
If $\mathcal S_{u_+}\cup\mathcal S_{u_-}\neq\emptyset$, then Alternative~2 holds.
We first assume that $S_{u_+}\neq \emptyset$. In this case, for any $\om \subset\subset \Om\setminus \mathcal{S}_{u_+}$ we have
\[
\limsup_{n\rightarrow +\infty}\|u_{+,n}^+\|_{L^\infty(\om)}<+\infty,
\]
and therefore, there exists $s_+\in L^\infty_{loc}(\Om \setminus \mathcal S_{u_+})$ such that $\nu_{{{+,n}}|_{\om}}\rightarrow s_+$ in $L^p(\om)$ for all $p\in[1,+\infty)$. It follows that $\nu_{{{+}}|_\om}= s_+ dx$, while the singular part of $\nu_+$ is supported on $\mathcal S_{u_+}.$ Hence,
\[
\nu_+= s_+ + \sum_{p\in \mathcal S_{u_+}}n_{+,p}\delta_p
\]
for some $n_{+,p}\geqslant 4\pi.$ Similarly

\[
\nu_-= s_- +\sum_{p\in \mathcal S_{u_-}}n_{-,p}\delta_p
\]
where $n_{-,p} \geqslant 4\pi.$
\par
We are left to show that
\[
\mathcal{S}_{u_+}= \mathcal S_+\qquad \mbox{ and } \qquad \mathcal S_{u_-}= \mathcal S_-.
\]
Let us start by proving that $\mathcal S_+ \subseteq \mathcal S_{u_+}.$ To this aim assume $p_0 \not\in \mathcal S_{u_+}$.
Then, by Lemma~\ref{lem:bmonmflds} there exists a neighborood of $p_0$ $U\subset \Om$ such that
\[
\limsup_{n\rightarrow \infty} \|u_{+,n}^+\|_{L^\infty(U)} < +\infty.
\]
Since $v_n= u_{+,n}- u_{-,n} \leqslant u_{+,n}+ C$, this implies
\[
\limsup_{n\rightarrow \infty} \| v_{n}^+\|_{L^\infty(U)} < +\infty
\]
i.e. $p_0\not \in \mathcal S_+.$
To prove that $\mathcal S_{u_{+}} \subseteq \mathcal S_+$ let $p_0\in \mathcal S_{u_+}$. As already seen $S_{u_+} $ coincides with the singular support of $\nu_+$ and consequently the sequence of functions
\[
\nu_{+,n} = \lambda_n\int_{I_+}V(\alpha,x,v_n)e^{\alpha v_n}\calP(d\alpha)dx
\]
is $L^\infty$-unbounded near $p_0\in \mathcal S_{u_+}$. This implies that, for every $r > 0$
\begin{equation*}
+ \infty = \lim_{n\rightarrow \infty}\sup_{B(p_0,r)} \nu_{+,n}\leqslant \lim_n \sup_{B(p_0,r)}C_1\lambda_n(e^{v_n}+1).
\end{equation*}
In particular
 \[
 \lim_{n\rightarrow \infty}\sup_{B(p_0,r)} v_{n}=+\infty
 \]
so that $p_0\in\mathcal S_+.$ The proof for $\mathcal S_-$ is similar.
\par
In order to prove \eqref{limiteq} we generalize the approach in \cite{os}.
Let $k_n(\alpha,x)=V(\alpha,x,v_n(x))$.
In view of (V2), we have $\|k_n\|_{L^\infty(I\times\Om)}\le C_1$.
Therefore, passing to a subsequence, we may assume that
$k_n$ converges weak-$\ast$ in $L^\infty(I\times\Om)$ to some $k\in L^\infty(I\times\Om)$.
Setting
\[
c_n=\frac{\lambda_n}{|\Om|}\iint_{I\times\Om}V(\alpha,x,v_n)e^{\alpha v_n},
\]
in view of (V3) we may assume that $c_n\to c\in\R$.
On the other hand, since $v_n$ is bounded in $W^{1,q}(\Om)$
for all $q\in[1,2)$, we may also assume that $v_n\to v\in W^{1,q}(\Om)$
strongly in $L^r(\Om)$ for $r\in[1,\infty)$.
We fix $\om\Subset\Om\setminus\mathcal S$ and we take a test function $\varphi\in C^\infty(\om)$.
We have, for all $n$,
\begin{equation}
\int_\Om\nabla v_n\cdot\nabla\varphi
=\lambda_n\iint_{I\times\Om}k_n(\alpha,x)e^{\alpha v_n}\varphi\,\Pda dx-c_n\int_\Om\varphi.
\end{equation}
Taking limits, we obtain
\begin{equation}
\int_\Om\nabla v\cdot\nabla\varphi
=\lambda_0\iint_{I\times\Om}k(\alpha,x)e^{\alpha v}\varphi\,\Pda dx-c_0\int_\Om\varphi.
\end{equation}
Since $\varphi$ is an arbitrary test function supported in $\om$, we conclude that
\eqref{limiteq} holds true in $\om$.
Since $\om\Subset\Om\setminus\mathcal S$ is also arbitrary,
\eqref{limiteq} is established on the whole of $\Om$.
\end{proof}
We proceed towards the proof of Theorem~\ref{thm:secondblowup}.
Hence we assume that depends on $\alpha,v$ only,  $\nabla_xV(\alpha,v)=0$
and that (V2')--(V3') hold.
We denote
\[
\tilde\mu_{\pm,n}(dx)=\lambda_n\int_{I_\pm}\frac{V(\alpha,v_n)}{\alpha}e^{\alpha v_n}\calP(d\alpha) dx.
\]
Since $\tilde \mu _{\pm,n}(\Om)\leqslant C_2'\lambda_n$,
up to subsequences $\tilde \mu_{\pm,n}\stackrel{*}{\rightharpoonup} \tilde \mu _{\pm}$ for some Borel measures $\tilde \mu_\pm\in\mathcal M(\Om).$
We first prove a lemma.
\begin{lem}
There exists $\tilde s_{\pm} \in L^{\infty}_{loc}(\Om\setminus \mathcal S_{\pm})$ and $\tilde m_\pm(p)\geqslant 4\pi$, $p\in \mathcal S_\pm$, such that
\begin{equation}\label{23bis}
\tilde \mu_\pm = \tilde s_\pm + \sum_{p\in\mathcal S_\pm} \tilde m_\pm(p)\delta_p.
\end{equation}
\end{lem}
\begin{proof}
By definition of $\mathcal S_\pm$, for every $\om \subset\subset \Om\setminus \mathcal S_\pm$ there exists a positive constant $C = C (\om)$ such that
\[
\sup_\om|v_n|\leqslant C\qquad \qquad \mbox{ for any }n\in \mathbb N.
\]
It follows that, for any measurable set $E \subset  \om$
\[
\tilde\mu_{\pm,n}(E)=\lambda_n\iint_{I_\pm\times E} \frac{V(\alpha,v_n)}{\alpha}e^{\alpha v_n }\calP(d\alpha)
\leqslant C_1'\lambda_ne^C |E|.
\]
Hence, the singular parts of $\tilde \mu_\pm$ are contained in $\mathcal S_\pm$ so that \eqref{23bis} holds for some $\tilde s_\pm\in L^1 (\Om )\cap L^\infty_{loc}(\Om\setminus \mathcal S_\pm)$ and for some  $\tilde m_\pm ( p ) \geqslant 0$, $p \in \mathcal S_\pm$. On the other hand, since $\tilde \mu_{\pm,n}\geqslant \nu_{\pm,n}$, then $\tilde m_\pm (p) =\tilde \mu_\pm(\{p\})\geqslant \nu_+ (\{p\})\geqslant 4\pi$. This completes our proof.
\end{proof}
Let $\mu_n$ and $\mu$ be as in formulas \eqref{mu-n} and \eqref{mu}. We are in position to prove Part~(i) of Theorem~\ref{thm:secondblowup}:
\begin{proof}[Proof of Theorem~\ref{thm:secondblowup}]
Part (i).
To prove that there exists $\zeta_p \in \mathcal M(I)$ and $r \in L^1 (I\times \Om),$ $r\geqslant 0$ such that
\[
\mu(d\alpha dx)= \sum_{p\in\mathcal S}\zeta_p(d\alpha) \delta_p (dx) + r(\alpha,x ) \calP(d\alpha) dx,
\]
it suffices to show that the singular part of $\mu$ is supported on $I\times\mathcal S$. To this aim, let us take
$A\Subset I\times (\Om\setminus \mathcal S )$. Then, there exists a constant $C = C(A)$
such that $\|\alpha v_n\|_{L^\infty(A)}\leqslant C$. Hence, for large values of $n$ we
obtain
\[
\mu_n(A)=\lambda_n\iint_{A}\frac{V(\alpha,v_n)}{\alpha}e^{\alpha v_n}\leqslant C_1'(\lambda_0 +1 )e^C\iint_A\Pda dx
\]
so that, on $A$, $\mu_n$ is absolutely continuous. This implies that $\mu_n$ does not have singularities on
$A\Subset I \times (\Om\setminus \mathcal S ) $ so that the thesis follows.
\par
Part (ii).
We recall from Section~\ref{sec:results} that
\[
\mu_\alpha = \lambda \frac{V(\alpha,v)}{\alpha} e^{\alpha v}\,dx.
\]
We define
\[
u_\alpha(x)= G\star \mu_\alpha(x)=\int_\Om G(x,x')\mu_\alpha(x')dx',
\]
where $G$ is the Green's function defined by \eqref{green}. Then,
\[
v =\int_I\alpha u_\alpha \calP(d\alpha)
\]
and $(u_\alpha)_{\alpha\in I}$ satisfies the Liouville type system:
\begin{align*}
&-\Delta u_\alpha=\lambda\frac{V(\alpha,v)}{\alpha}\exp\left\{\alpha\int_I\alpha'u_{\alpha'}\calP(d\alpha')\right\}-c_{\alpha}
&&\int_\Om u_\alpha=0,
\end{align*}
where
\[
c_{\alpha}=\frac{\lambda}{|\Om|}\int_\Om\frac{V(\alpha,v)}{\alpha}\exp\left\{\alpha\int_I\alpha'u_{\alpha'}\calP (d\alpha')\right\}.
\]
Now we use Suzuki's symmetry argument as introduced in \cite{ss,os2,sbook}.
Let us first observe that $\mu_\alpha$ verifies:
\begin{equation}
\label{nablamua}
\begin{split}
\nabla \mu_\alpha = \alpha\mu_\alpha\nabla v
=\alpha \mu_\alpha \int_I \alpha' \nabla u_{\alpha'}\calP(d\alpha')=\alpha \mu_\alpha \int_I \alpha' (\nabla G) \star \mu_{\alpha'}\calP(d\alpha').
\end{split}
\end{equation}
We note that, despite of the general form of the potential $V$,
equation~\eqref{nablamua} is identical to equation~(30) in \cite{ORS}.
Equation~\eqref{nablamua} and with Part~(i) in Theorem~\ref{thm:secondblowup} are key ingredients
necessary to the above mentioned symmetry argument.
With such ingredients at hand, the proof of Part~(ii) follows exactly as in \cite{ORS}.
For the reader's convenience, we sketch it briefly in what follows.
\par
Let $\chi $  be a $C^1$-vector field over $\Om$, and define
\[
\rho_\chi  : \Om^2 \setminus \{ (x,x')\in \Om^2: x=x'\} \rightarrow \mathbb R
\]
by
\[
\rho_\chi(x,x')=\frac12 [\chi(x)\cdot \nabla_x G(x,x') + \chi (x')\cdot  \nabla_{\chi'} G(x,x')].
\]
Recall from Section~\ref{sec:results} that
\[
\mu_\alpha^n = \lambda_n \frac{V(\alpha,v_n)}{\alpha} e^{\alpha v_n}\,dx.
\]
Then, Suzuki's symmetry trick yields the following key identity:
\begin{equation*}
\iint_{I\times\Om} (\dive \chi)\mu^n_\alpha \calP(d\alpha)dx =
-\iint_{I^2}  \alpha \alpha' \calP(d\alpha)\calP(d\alpha')\int \int_{\Om^2} \rho_\chi (x,x')\mu^n_\alpha\mu^n_{\alpha'} dxdx',
\end{equation*}
For any choice of $\chi$ such that $\rho_\chi$ is continuous on $\Om^2$, taking limits in last equality,
in view of Part (i) we obtain the identity:
\begin{equation}
\label{34}
\begin{split}
\sum_{p\in\mathcal S} &\int_{I} (\dive \chi)(p) \zeta_p (d\alpha) + \iint_{I\times \Om }  (\dive \chi)(x)r (\alpha,x)\calP(d\alpha)dx \\
& = \iint_{I^2} \big[\sum_{p,q \in \mathcal S} \zeta_p (d\alpha)\zeta_q (d\alpha')\rho_\chi(p,q)
 + \sum_{p\in \mathcal S} \zeta_p(d\alpha)  \calP(d\alpha')\int_\Om r(\alpha',x')\rho_\chi(p,x') dx' \\
& +  \sum_{q\in \mathcal S} \zeta_q(d\alpha')  \calP(d\alpha) \int_\Om r(\alpha,x) \rho_\chi(x,q) dx
+ \iint_{\Om^2} r(\alpha,x)r(\alpha',x') \rho_\chi(x,x') dx dx' \calP(d\alpha) \calP(d\alpha')\big].
\end{split}
\end{equation}
We fix $p_0\in \mathcal S$
and take an isothermal coordinate chart $(\psi, U)$ satisfying $\psi(p_0) = 0,$ $g(X) = e^\xi (dX^2_1 + dX_2^2)$, and
$\xi(0) = 0$. Let $B(p_0, 2r)\subset U$ and $B(p_0, 2r) \cup\mathcal S = \{p_0\}$.
We recall the following expansions of the Green's function:
\[
G(X, X')=-\frac{1}{2\pi}\log|X-X'| + \om (X, X'),
\]
\[
\nabla_X G(X,X')= -\frac{1}{2\pi} \frac{X-X'}{|X-X'|^2} +\nabla_X \om (X, X'),
\]
\[
\nabla_{X'} G(X,X')= \frac{1}{2\pi} \frac{X-X'}{|X-X'|^2} +\nabla_{X'} \om (X, X'),
\]
with $\om$ satisfying

\[
\| \om \|_{L^\infty(B(p_0,2r)^2)} +\| \nabla_X \om\|_{L^\infty(B(p_0,2r)^2)} + \| \nabla_{X'}\om \|_{L^\infty(B(p_0,2r)^2)}  = O(1)
\]
as $r\rightarrow 0$. Let $\fip\in C(\Om)$ be a cut-off function such that $\fip \equiv 1$ in $B(p_0, r)$ and $\fip\equiv 0$ in $\Om\setminus  B(p_0, 2r)$.
We choose $\chi (X) = 2X\fip $. With this choice of $\chi$ we may write:
\[
\rho_\chi(X,X')= ( -\frac{1}{2\pi} + \eta) \fip
\]
where $\eta (X, X')$ is a continuous function on $\Om^2$. Moreover, we have
\[
\dive \chi (X) = |g|^{-1/2}\de X_j (|g|^{1/2}(\chi)^j )  = 4+ O(X).
\]
Consequently, taking limits for each term in \eqref{34} as $r\downarrow 0$ we derive:
\begin{align*}
&\sum_{p\in\mathcal S}\int_{I}(\dive\chi)(p)\zeta_{p_0}(d\alpha)\rightarrow 4\int_I \zeta_{p_0}(d\alpha);\\
&\left|\iint_{I\times \Om } (\dive \chi)(x)r (\alpha,x)\calP(d\alpha)dx \right|=o(1);\\
&\iint_{I^2}\sum_{p,q\in \mathcal S}\rho_\chi(p,q)\zeta_p(d\alpha)\zeta_q(d\alpha')\rightarrow -\frac{1}{2\pi}\iint_{I^2}\zeta_{p_0}(d\alpha)\zeta_{p_0}(d\alpha');\\
&\iint_{I^2}\sum_{p\in \mathcal S}\zeta_p (d\alpha) \int_\Om r(\alpha',x')\rho_\chi(p,x')dx' \calP(d\alpha')= o(1).
\end{align*}
Similarly, we have:
\begin{align*}
&\iint_{I^2}  \sum_{q\in \mathcal S}  \zeta_q(d\alpha') \int_\Om r(\alpha,x) \rho_\chi(x,q)\calP(d\alpha)=o(1)\\
&\iint_{I^2}\iint_{\Om^2} r(\alpha,x)r(\alpha',x') \rho_\chi(x,x') dx dx'  \calP(d\alpha) \calP(d\alpha')=o(1).
\end{align*}
Inserting into \eqref{34} we conclude the proof of Part~(ii).
%
\par
Part~(iii).
We provide the proof for the $``I_+ -$case", the proof for $I_-$ being exactly the same. Let $\ee > 0$  and let $\fip \in C(\Om)$, $\psi \in C(I)$, $0\leqslant\psi(\alpha)\leqslant 1$, $\psi \equiv 1$ on $I_+$, $\psi \equiv 0$ on $[-1, -\ee]$. We have
\begin{equation}\label{equality}
\begin{split}
\iint_{I\times \Om} |\alpha| \fip(x)\psi(\alpha)
\mu_n(d\alpha dx)
&= \int_\Om \fip(x)\nu_{+,n}(dx) \\
&+ \lambda_n\iint_{[-\ee,0]\times\Om}V(\alpha,v_n)\psi(\alpha)\varphi(x)\,dx\calP(d\alpha).\\
\end{split}
\end{equation}
Taking limits on the left-hand side of  \eqref{equality} as $n\rightarrow\infty$, by \eqref{zita} we have
\begin{equation*}
\begin{split}
&\iint_{I\times \Om} |\alpha| \fip(x)\psi(\alpha)
\mu_n(d\alpha dx)\\
&\rightarrow \sum_{p\in \mathcal S} \int_I |\alpha|\psi (\alpha) \fip(p) \zeta _p(d\alpha)  + \iint_{I\times\Om}|\alpha|\fip(x)\psi(\alpha) r(\alpha,x) \calP(d\alpha) dx.
\end{split}
\end{equation*}
Moreover,
\[
 \int_I |\alpha|\psi (\alpha)\zeta_p(d\alpha)=\int_{I_+}|\alpha| \zeta _p(d\alpha)+ \int_{[-\ee,0]}|\alpha|\psi (\alpha) \zeta _p(d\alpha),
\]
and
\begin{equation*}
\begin{split}
&\iint_{I\times\Om}|\alpha|\fip(x)\psi(\alpha) r(\alpha,x) \calP(d\alpha) dx\\
& = \iint_{I_+\times\Om}|\alpha|\fip(x)  r(\alpha,x) \calP(d\alpha) dx+ \iint_{[-\ee,0]\times\Om}|\alpha|\fip(x)\psi(\alpha) r(\alpha,x) \calP(d\alpha) dx.
\end{split}
\end{equation*}
We note that
\[
0\leqslant \int_{[-\ee,0]}|\alpha|\psi (\alpha) \zeta _p(d\alpha)
\leqslant\ee\int_{[-\ee,0]}\psi(\alpha)\zeta_p\le c_1\ee.
\]
Furthermore,
\[
\left|\iint_{[-\ee,0]\times\Om}|\alpha|\fip(x)\psi(\alpha) r(\alpha,x) \calP(d\alpha) dx\right|\leqslant \ee\|\fip\|_{\infty}\iint_{I\times\Om}r(\alpha,x)\calP(d\alpha) dx.
\]
Analogously, by passing to the limit as $n\rightarrow\infty$ on the right hand side of  \eqref{equality} we have
\[
\int_\Om \fip(x)\nu_{+,n}(dx) \rightarrow \sum_{p\in \mathcal S_+} n_{+,p} \fip(p)+\int_\Om s_+\fip.
\]
Moreover, in view of (V2'),
\[
\lambda_n \int_{[-\ee,0]}|\alpha| \psi(\alpha)
\int_{\Om}\frac{|V(\alpha,v_n)|}{|\alpha|}\fip(x)e^{\alpha v_n}\,dx \calP(d\alpha)
\leqslant C_2'\lambda_n\ee \|\fip\|_\infty.
\]
Hence, combining the estimates above we obtain
\begin{equation*}
\begin{split}
&\sum_{p\in \mathcal S_+} n_{+,p} \fip(p)+\int_\Om s_+\fip + c_1 \ee \|\fip\|_\infty \\
&=  \sum_{p\in \mathcal S} \int_{I_+} |\alpha| \zeta _p(d\alpha) \fip(p) + \iint_{I_+\times\Om}|\alpha|\fip(x) r(\alpha,x) \calP(d\alpha) dx+ c_2 \ee \|\fip\|_\infty,
\end{split}
\end{equation*}
where $c_1$ and $c_2$ are constants uniformly bounded with respect to $\ee$. So, by passing to the limit as $\ee \rightarrow 0^+ $ in last equality, we obtain
\begin{equation}\label{25bis}
\sum_{p\in \mathcal S_+} n_{+,p} \fip(p)+\int_\Om s_+\fip = \sum_{p\in \mathcal S} \int_{I_+} |\alpha| \zeta _p(d\alpha) \fip(p) + \iint_{I_+\times\Om}|\alpha|\fip(x) r(\alpha,x) \calP(d\alpha) dx.
\end{equation}
Now, assume $\fip\in C (\Om)$ with supp$\fip \subset \Om\setminus \mathcal S.$ Then,
\[
\int_\Om s_+\fip = \int_\Om \fip(x) \int_{I_+} |\alpha|r(\alpha,x) \calP(d\alpha) dx
\]
so that for almost every $x\in\Om$
\[
s_+ =  \int_{I_+} |\alpha|r(\alpha,x) \calP(d\alpha)
\]
since $\mathcal S$ is null set with respect to $dx$.  By \eqref{25bis} this implies
\begin{equation}\label{26bis}
\sum_{p\in \mathcal S_+} n_{+,p} \fip(p) = \sum_{p\in \mathcal S} \int_{I_+} |\alpha| \zeta _p(d\alpha) \fip(p).
\end{equation}
Now let us  fix $p_0\in \mathcal S_+$ and let $\fip\in C(\Om)$ be such  that supp$\fip\subset B_\rho(p_0)$, with $B_\rho(p_0) \cap \mathcal S =\{p_0\}$ and verifying $\fip(p_0) = 1$. By \eqref{26bis} then we have

\[
n_{+,p_0} = \int_{I_+}|\alpha| \zeta _{p_0}(d\alpha)
\]
for any $p_0 \in \mathcal S_+$. To conclude, for $p_0 \in\mathcal S_-\setminus \mathcal S_+$, let us assume $\fip\in C(\Om)$ as above. By \eqref{26bis} we get
\[
\int_{I_+}|\alpha| \zeta _{p_0}(d\alpha)=0.
\]
This completes our proof.
\end{proof}
\section{The cases of physical interest}
\label{sec:neri}
In this section we consider the special cases of \eqref{genmf} which are of
interest in statistical turbulence.
Namely, we consider the special case $V(\alpha,x,v)=V_1(\alpha,v)$, where $V_1$ is given by
\eqref{VSS}, corresponding to Sawada and Suzuki's equation \eqref{SawadaSuzuki},
and the special case $V(\alpha,x,v)=V_2(\alpha,v)$, where $V_2$ is given by \eqref{neripotential},
corresponding to Neri's equation \eqref{NE}.
\par
It is clear that in both cases $V_1,V_2$ satisfy (V0)--(V1).
We claim that (V2')--(V3') are also satisfied.
Indeed, by Jensen's inequality we have
\[
\int_\Om e^{\alpha v}\ge|\Om|
\]
for all $v\in\mathcal E$ and for all $\alpha\in I$.
Therefore, we have
\begin{align*}
0\le\alpha^{-1}V_1(\alpha,v)=&\frac{1}{\int_\Om e^{\alpha v}}\le\frac{1}{|\Om|}\\
0\le\alpha^{-1}V_2(\alpha,v)=&\frac{1}{\iint_{I\times\Om} e^{\alpha v}}\le\frac{1}{|\Om|},
\end{align*}
where we used $\calP(I)=1$ in the last inequality. Hence (V2') is satisfied with $C_1'=|\Om|^{-1}$
in both cases.
On the other hand, we have
\begin{align*}
\iint_{I\times\Om}|V_1(\alpha,v)|e^{\alpha v}\,dx\Pda
=&\iint_{I\times\Om}\frac{|\alpha|e^{\alpha v}}{\int_\Om e^{\alpha v}}\,dx\Pda
=\int_I|\alpha|\Pda\le1\\
\iint_{I\times\Om}|V_2(\alpha,v)|e^{\alpha v}\,dx\Pda
=&\iint_{I\times\Om}\frac{|\alpha|e^{\alpha v}}{\iint_{I\times\Om}e^{\alpha v}}\,dx\Pda\le1.
=\int_I|\alpha|\Pda\le1
\end{align*}
Hence, (V3') is also satisfied in both special cases with $C_2'=1$.
We conclude that Theorem~\ref{thm:firstblowup} and Theorem~\ref{thm:secondblowup}
hold true for solution sequences to \eqref{SawadaSuzuki} and \eqref{NE}.
In other words, \eqref{SawadaSuzuki} and \eqref{NE} are similar from the point of view
of blow-up.
\par
If $k\equiv0$ in \eqref{limiteq} we say that residual vanishing occurs.
In the following theorem we provide a sufficient condition for residual vanishing,
in the special case where $V=V_2$ has the form \eqref{neripotential}.
The proof is an adaptation of an argument from \cite{os3} to our case.
\begin{thm}
\label{thm:residual}
If $\mathrm{supp}\calP\cap\{-1,1\}\neq0$ and if there exists
$p\in \mathcal S_+\setminus \mathcal S_-$ such that $n_{+,p}>4\pi$ then $k\equiv0$ in \eqref{limiteq}.
\end{thm}
\begin{proof}
We consider the case where $\calP([1-\delta,1] )>0$
for all $0<\delta\ll1$ and $p\in \mathcal S_+\setminus \mathcal S_-.$
The remaining cases are is analogous. For every fixed $T>0$ we truncate the Green's function
\[
G^T(x,\cdot)=\min\{T, G(x,\cdot)\}\in C(\Om).
\]
Then,
\begin{equation*}
\begin{split}
u_{+,n}(x)= \int_\Om G(x,\cdot ) \nu_{+,n}
\geqslant \int_\Om G^T(x,\cdot) \nu_{+,n} &\rightarrow \int_\Om G^T(x,\cdot)\nu_+\\
&= n_{+,p}G^T(x,p)+\int_\Om G^T(x,\cdot)(\nu_{+}-n_{+,p}\delta_p)\\
&\geqslant n_{+,p}G^T(x,p)-C.\\
\end{split}
\end{equation*}
Hence,
\[
\liminf_{n} u_{+,n}(x)\geqslant n_{+,p}G^T(x,p)-C.
\]
Letting $T\rightarrow \infty,$
\[
\liminf_n u_{+,n}\geqslant n_{+,p}G(x,p)-C.
\]
 On the other hand it is well known that in a local chart
on $B_\rho=B_\rho(p)$
\[
G(x,p)\geqslant \frac 1{2\pi}\log\frac{1}{|x|} - C.
\]
Therefore,
 \[
 \exp\{\alpha u_{+,n}(x)\}\geqslant \exp\{\alpha n_{+,p}(\frac 1{2\pi}\log\frac{1}{|x|} - C)\}
\simeq\left(\frac{1}{|x|}\right)^{\alpha n_{+,p}/2\pi}.
 \]
Since $p\not\in\mathcal S_-$, then $u_-(x)\leqslant C$ in $B_\rho(p)$ whenever $\rho $ is suitable small.
We observe that the function $\int_\Om e^{\alpha v}dx$ is increasing with respect to $\alpha >0$.
In fact, differentiation with respect to $\alpha$ yields:
%

\begin{equation}
\label{increase}
 \begin{split}
\frac{d}{d\alpha}\int_\Om e^{\alpha v}dx &= \int_\Om  v e^{\alpha v}dx
\geqslant \int_{v\geqslant 0}  v dx - \int_{v< 0} e^{\alpha v}(-v)dx \\
&> \int_{v\geqslant 0}  v dx - \int_{v< 0} (-v)dx = \int_\Om v=0.
\end{split}
 \end{equation}
Using this fact, we conclude that,
 \begin{equation*}
 \begin{split}
 \liminf_n\iint_{I\times \Om} e^{\alpha v_n} &=\liminf_n\iint_{I\times \Om } e^{\alpha (u_{+,n}-u_{-,n})}\\
 & \geqslant\calP([1-\delta,1]) \liminf_n\int_{\Om }e^{(1-\delta)(u_{+,n}-u_{-,n})} \\
&\geqslant\calP([1-\delta,1]) e^{-C}\liminf_n\int_{B_\rho(p)}e^{ (1-\delta)u_{+,n}}\\
&\geqslant e^C\calP([1-\delta,1])\int_{B_\rho(p)}\left(\frac{1}{|x|}\right)^{(1-\delta)n_{+,p}/2\pi}.\\
 \end{split}
 \end{equation*}
Choosing $\delta$ such that $(1-\delta)n_{+,p}>4\pi$ we conclude the proof.
\end{proof}
We now consider the Trudinger-Moser type inequalities associated to \eqref{SawadaSuzuki} and \eqref{NE}.
We recall that \eqref{SawadaSuzuki} is the Euler-Lagrange equation for the functional
\begin{equation}
\label{SSfunct}
\mathcal J_{\lambda}(v)=\frac{1}{2}\|\nabla v\|_2^2-\lambda\int_I\log\left(\int_\Om e^{\alpha v}\right)\,\Pda,
\end{equation}
defined for $v\in\mathcal E$.
For $\calP=\delta_1$, the functional $\mathcal J_{\lambda}(v)$ is the functional
\begin{equation}
\label{standTM}
\frac{1}{2}\|\nabla v\|_2^2-\lambda\log\left(\int_\Om e^{v}\right),
\end{equation}
whose Euler-Lagrange equation is the standard meanfield equation \eqref{standmfe}.
In view of the classical Trudinger-Moser inequality, as established in \cite{fon}:
\begin{equation}
\label{fontana}
\sup\left\{\int _\Om e^{4\pi v^2}:\ v\in\mathcal E,\ \|\nabla v \|_2\leqslant 1\right\}<+\infty,
\end{equation}
where the constant $4\pi$ is sharp,
the functional~\eqref{standTM} is bounded from below on $\mathcal E$ if and only if
$\lambda\le8\pi$.
In \cite{os3}, as an application of the blow-up analysis, it is shown that if
$\calP=t\delta_1+(1-t)\delta_{-1}$, $t\in[0,1]$, then the optimal value of $\lambda$
which ensures boundedness from below of the functional \eqref{SSfunct}
is improved to $8\pi\min\{t^{-1},(1-t)^{-1}\}$.
By using the blow-up analysis developed in \cite{ORS} and similar arguments one may
check that \eqref{SSfunct} is bounded from below if
\[
\lambda\le\frac{8\pi}{\max\{\int_{I_+}\alpha^2\Pda,\int_{I_-}\alpha^2\Pda\}}.
\]
This value is however is in general not optimal.
The best constant is actually given by
\[
\inf\left\{\frac{8\pi\calP(K_\pm)}{\left(\int_{K_\pm}\alpha\Pda\right)^2}:\ K_\pm\subset I_\pm\cap\mathrm{supp}\calP\right\},
\]
see \cite{RS}.
\par
In view of such ``improved" Trudinger-Moser inequalities, it is natural to
seek analogous results for \eqref{NE}.
More precisely, we note that \eqref{NE} is the Euler-Lagrange equation for the
functional
\begin{equation}
\label{Nerifunct}
\mathcal K_\lambda(v)=\frac{1}{2}\|\nabla v\|_2^2-\lambda\log\left(\iint_{I\times\Om}e^{\alpha v}\,dx\Pda\right).
\end{equation}
However, it is not difficult to check that such an improvement does \textit{not} hold for
$\mathcal K_\lambda$, that is, $\mathcal K_\lambda$ is bounded from below if and only if $\lambda\le8\pi$.
The ``if" part was already observed in \cite{ne}.
Indeed, from
\begin{equation*}
\alpha v \leqslant \frac{\|\nabla v\|_2^2}{16 \pi}+4\pi \alpha^2 \frac{v^2}{\|\nabla v\|_2^2}
\end{equation*}
we derive using \eqref{fontana} that
\begin{equation*}
\log\left(\int_{\Om}  \int_{[-1,1]}e^{\alpha v}\calP(d\alpha) dx\right) \leqslant \frac1{16\pi} \|\nabla v\|_2^2+K,
\end{equation*}
where $K$ is independent of $v\in\mathcal E$. See also \cite{s1book}.
Therefore $\mathcal K_\lambda$ is bounded below
if $\lambda\le8\pi$.
On the other hand, differently from what happens for the functional \eqref{SSfunct},
the value $8\pi$ is also optimal,
provided that $\mathrm{supp}\calP\cap\{-1,1\}\neq\emptyset$.
Indeed, the following holds:
\begin{thm}
\label{thm:TM}
Let $\mathrm{supp}\calP\cap\{-1,1\}\neq\emptyset$.
Then, the functional $\mathcal K_\lambda(v)$is bounded from below on $\mathcal E$ if and only if $\lambda \leqslant8\pi.$
\end{thm}
\begin{proof}
We need only prove that
\begin{equation}\label{opt}
\inf_{v\in \mathcal E}\mathcal K_\lambda(v)=-\infty, \qquad \quad\forall \lambda >8\pi.
\end{equation}
Using \eqref{increase}, for any $\delta>0$ we have:
\begin{eqnarray*}
\mathcal K_\lambda (v) &&= \frac 12 \|v\|^2-\lambda \log\left(\iint_{I\times\Om} e^{\alpha v}\calP (d\alpha )dx\right)\\
&&= \frac 12 \|v\|^2-\lambda \log\left( \int_{1-\delta}^1\int_\Om  e^{\alpha v}\calP(d\alpha)  dx
+\int^{1-\delta}_{-1}\int_\Om  e^{\alpha v}\calP(d\alpha)  dx\right)\\
&&= \frac 12 \|v\|^2-\lambda \log\left( \int_{1-\delta}^1\int_\Om  e^{\alpha v}\calP(d\alpha)  dx \right)   \\
&&-\lambda \log\left( 1 + \frac {\int_{-1}^{1-\delta}\int_{\Om}e^{\alpha v}\calP (d\alpha )dx}{ \int_{1-\delta}^1\int_\Om  e^{\alpha v}\calP(d\alpha)  dx} \right)\\
&&\leqslant \frac 12 \| v\|^2-\lambda \log\left(\int_\Om  e^{(1-\delta) v} dx\right)  -\lambda \log (\calP([1-\delta,1]))\\
&&-\lambda \log  \left( 1 + \frac {\int_{-1}^{1-\delta}\int_{\Om}e^{\alpha v}\calP (d\alpha )dx}{ \int_{1-\delta}^1\int_\Om  e^{\alpha v}\calP(d\alpha)  dx} \right)\\
&&\leqslant \frac{1}{(1-\delta)^2} \left[\frac 12 \|(1-\delta) v\|^2-\lambda (1-\delta)^2 \log\left(\int_\Om  e^{(1-\delta) v} dx\right) \right] -\lambda \log (\calP([1-\delta,1])). \\
\end{eqnarray*}

Hence, for $\lambda (1-\delta)^2>8\pi$, the right hand side of last inequality is unbounded from below (see \cite{fon}) and so
\[
\inf_{v\in\mathcal E}\mathcal K_\lambda (v) =-\infty  \qquad \mbox{ for any }\lambda >\frac{8\pi}{(1-\delta)^2}.
\]

Since $\delta>0 $ is arbitrary, \eqref{opt} follows.
\end{proof}

\section{Acknowledgements}
We thank Professor Takashi Suzuki for suggesting the problem, for sharing many interesting discussions
with us and for constantly supporting our research.
This work is partially supported by the Marie Curie International Researchers Exchange Scheme
FP7-MC-IRSES-2009-247486 of the Seventh Framework Programme.

\end{document}